\documentclass[10 pt, leqno]{amsart}

\usepackage {amsfonts}
\usepackage{amsthm}
\usepackage{amssymb}
\usepackage{latexsym}
\usepackage{amsmath}
\usepackage{mathrsfs}

\theoremstyle{plain}
\newtheorem{tw}{Theorem}

\theoremstyle{definition}

\newcommand{\bn}{\Bbb N}

\newcommand{\bz}{\Bbb Z}

\newcommand{\alg} {\mathsf{A}}

\newcommand {\hte} {{\textup{ht}}}

\newcommand{\clg} {\mathsf{C}}

\newcommand{\blg}{\mathsf{B}}

\newcommand{\mlg}{\mathsf{M}}

\newcommand {\Aut} {{\textrm{Aut}}}

\newcommand{\tu}{\textup}

\newcommand{\Kil}{\mathsf{K}}

\newenvironment{rlist}
{

\begin{enumerate}}
{\end{enumerate}}

\newcommand{\ot}{\otimes}

\newcommand{\wt}{\widetilde}

\begin{document}

\author{Adam Skalski}
\footnote{\emph{Permanent address of the author:} Department of Mathematics, University of \L\'{o}d\'{z}, ul. Banacha 22, 90-238
\L\'{o}d\'{z}, Poland.}
\address{Department of Mathematics and Statistics, Lancaster University, Lancaster LA1 4YF, United Kingdom }
\email{a.skalski@lancaster.ac.uk}
\title{\bf On automorphisms of $C^*$-algebras whose Voiculescu entropy is genuinely noncommutative}


\begin{abstract}
\noindent We use the results of Neshveyev and St\o rmer to show that for a  generic shift on a $C^*$-algebra associated to a
bitstream the Voiculescu topological entropy is strictly larger that the supremum of topological entropies of its classical
subsystems.
\end{abstract}

\maketitle

Let $\alg$ be an (exact) $C^*$-algebra, $\alpha \in \Aut(\alg)$ and let $\hte(\alpha)$ denote the Voiculescu's topological
entropy of $\alpha$ (\cite{Voic}). The usual method of computing $\hte(\alpha)$ is based on two steps. First one produces an
explicit or semi-explicit approximating net through matrix algebras whose rank can be controlled and thus provides an estimate
from above (say given by $M>0$). Then one looks for an $\alpha$-invariant $C^*$-subalgebra $\clg \subset \alg$ such that
$\hte(\alpha|_{\clg})$ is known to be equal to $M$. Of course the second step is in a sense recurrent, as we need to be able to
determine the entropy of $\alpha|_{\clg}$. This becomes much easier when $\clg$ is a commutative algebra, as then
$\alpha|_{\clg}$ is induced by a homeomorphism $T$ of the spectrum of $\clg$ and it was shown in \cite{Voic} that
$\hte(\alpha|_{\clg})= h_{\tu{top}}(T)$ -- note that the general difficulty in understanding how the positive Voiculescu entropy
is produced is reflected in the fact that there is still no direct proof of the inequality $\hte(\alpha_T)\geq h_{\tu{top}}(T)$,
the argument of \cite{Voic} exploits the properties of dynamical state entropy and classical variational principle. Other
connections between the appearance of a non-zero noncommutative entropy and commutativity can be seen in \cite{HaS} (where the
occurrence of maximal entropy for a system of subalgebras is related to existence of suitable maximally abelian subalgebras) and
in \cite{free} (where free shifts are shown to have zero Voiculescu entropy).

One could therefore be tempted to believe that actually  \begin{equation}\hte(\alpha) =
\hte_c(\alpha),\label{coment}\end{equation} where
\[ \hte_c(\alpha) =  \sup \left\{\hte(\alpha|_{\clg}): \clg - \textrm{commutative } \alpha-\textrm{invariant } C^*-\textrm{subalgebra of  }
\alg\right\}.\]

The aim of this note is to observe that the bit-stream shifts investigated in Chapter 12  of \cite{book} (see also \cite{GolSto},
\cite{NST}) giving counterexamples to the tensor product formula for Connes-Narnhofer-Thirring entropy also provide
counterexamples to the equality in \eqref{coment}.

Before we state the result let us quickly summarise the properties of the bit-stream shifts  that we will need in what follows.
Let $X\subset \bn$. Then $X$ determines in a natural way a bicharacter $\omega$ on $G:=\bigoplus_{\bz} \bz_2$ and hence a
(nuclear) twisted $C^*$-algebra $A(X)=C^*(G, \omega)$ with a canonical faithful tracial state $\tau$. The algebra $A(X)$ is
equipped with a $\tau$-preserving shift-type automorphism $\sigma_X$. In Chapter 12 of \cite{book} it was shown that for almost
every (with respect to the Bernoulli measure on $2^{\bn}\approx\prod_{n \in \bn}\{0,1\}$) nonperiodic $X\subset \bn$ the
automorphism $\sigma_X$ has the following properties:
\begin{rlist}
\item $\tau$ is the unique $\sigma_X$-invariant state on $A(X)$;
\item $h_{\tau} (\sigma_X) = 0$;
\item $h_{\tau \ot \tau} (\sigma_X \ot \sigma_X) = \log 2$.
\end{rlist}

Here and in what follows $h_{\phi}(\alpha)$ denotes the noncommutative state entropy of Connes, Narnhofer and Thirring introduced
in \cite{CNT} ($\phi$ is an $\alpha$-invariant state).

\begin{tw} \label{thm}
For almost every nonperiodic set $X\subset \bn$ the binary shift $\sigma_X$ satisfies $\hte_c (\sigma_X)=0 < \frac{\log 2}{2}\leq
\hte(\sigma_X)$.
\end{tw}
\begin{proof}
Let $X$ be a  subset of $\bn$ such that $\sigma_X$ satisfies the conditions $\tu{(i)--(iii)}$ stated before the theorem. We will
write $\alg:=A(X)$, $\sigma:=\sigma_X$.

The upper estimate follows immediately from property (iii) above, subadditivity of Voiculescu's topological entropy with respect
to tensor products (Theorem 6.2.2 (iv) of \cite{book}) and the fact that it dominates the CNT entropy (Theorem 6.2.7 in
\cite{book}).

Let then $\clg$ be a commutative $\sigma$-invariant $C^*$-subalgebra of $\alg$. Write $\sigma_{\clg}$ for $\sigma|_{\clg}$. Due
to the
classical variational principle 
it suffices to show that $h_{\phi}(\sigma_{\clg})=0$ for every $\sigma_{\clg}$-invariant state $\phi$ on $\clg$. We claim first
that invariance of $\phi$ implies it is equal to $\tau|_{\clg}$. Indeed, suppose that $\psi$ is an arbitrary state on $\alg$
extending $\phi$. Let $\wt{\psi}$ be a weak$^*$-limit point of the sequence $(\frac{1}{n} \sum_{i=1}^n \psi \circ
\sigma^i)_{n=1}^{\infty}$. Then $\wt{\psi}$ is $\sigma$-invariant and extends $\phi$. By condition (i) above $\wt{\psi}=\tau$ so
$\phi=\tau|_{\clg}$.

We are almost done, as due to property (ii) before the statement of the theorem we have $h_{\tau}(\sigma)=0$. We need however to
remember that the CNT entropy is known to be monotone only under passing to the expected invariant subalgebras and we did not
assume that $\clg$ is expected. The traciality of $\tau$ tells us that we have a $\tau$-preserving expectation as soon as we pass
to the von Neumann algebraic setting -- this will be exploited in the last part of the proof.

Let $(\pi_{\tau}, L^2(\alg, \tau), \Omega_{\tau})$ be a GNS triple for $(\alg, \tau)$, let $\wt{\tau}$ denote the extension of
$\tau$ to $\pi_{\tau}(\alg)''$ (given by the vector state associated with $\Omega_{\tau}$) and let $\wt{\sigma}$ denote the
unique normal automorphism of $\pi_{\tau}(\alg)''$ satisfying $\wt{\sigma} (\pi_{\tau} (a)) = \pi_{\tau}(\sigma(a))$ for all $a
\in \alg$.  By Theorem 3.2.2 (ii) in \cite{book}

\begin{equation} \label{eq1}h_{\tau} (\sigma) = h_{\wt{\tau}} (\wt{\sigma}).\end{equation}
Note that $\wt{\sigma}$ leaves $\pi_{\tau}(\clg)''$ invariant.

Let $p\in B(L^2(\alg, \tau))$ be the orthogonal projection onto  $\Kil:=\textup{cl}(\pi_{\tau}(\clg) \Omega_{\tau})$. Then $p \in
\pi_{\tau}(\clg)'$. It is easy to see that $(\pi_{{\clg}}:=\pi_{\tau}|_{\clg}, \Kil, \Omega_{\tau})$ is a GNS triple for $(\clg,
\tau|_{\clg})$. Denote by $\hat{\tau}\in {\pi_{\clg}(\clg)''}_*$, $\hat{\sigma}_{\clg}\in \Aut(\pi_{\clg}(\clg)'')$ corresponding
normal extensions of $\tau|_{\clg}$ and $\sigma$. Again by Theorem 3.2.2 (ii) in \cite{book}
\begin{equation} \label{eq2} h_{\tau|_{\clg}} (\sigma_{\clg}) =
h_{\hat{\tau}} (\hat{\sigma}_{\clg}).\end{equation} It remains to observe that the map $\Psi: \pi_{\tau}(\clg)'' \to
\pi_{\clg}(\clg)'' $ given by $m \mapsto mp$ is a covariant isomorphism of $\pi_{\tau}(\clg)''$ and  $\pi_{\clg}(\clg)''$, by
which we mean that  $\Psi \circ \wt{\sigma}|_{\pi_{\tau}(\clg)''} = \hat{\sigma}_{\clg} \circ \Psi$ and $\hat{\tau} \circ\Psi=
\wt{\tau}|_{\clg}$. Indeed, the fact that $\Psi$ is injective follows from the faithfulness of $\tau$, whereas the surjectivity
and covariance properties are consequences of the normality of all maps involved and the fact that the corresponding statements
clearly hold on $\pi_{\tau}(\clg)$. Thus we have
\begin{equation} \label{eq3} h_{\wt{\tau}|_{\pi_{\tau}(\clg)''}}(\wt{\sigma}|_{\pi_{\tau}(\clg)''}) =
h_{\hat{\tau}} (\hat{\sigma}_{\clg}).\end{equation} As $\tau$ is a trace,
there exists
a (unique) $\tau$-preserving conditional expectation from $\pi_{\tau}(\alg)''$ onto $\pi_{\tau}(\clg)''$. Therefore by Theorem
3.2.2 (v) in \cite{book}
\begin{equation} \label{eq4} h_{\wt{\tau}|_{\pi_{\tau}(\clg)''}}(\wt{\sigma}|_{\pi_{\tau}(\clg)''}) \leq
h_{\wt{\tau}} (\wt{\sigma}).\end{equation} Putting together equations \eqref{eq1}-\eqref{eq4} yields
\[ h_{\tau|_{\clg}} (\sigma_{\clg}) \leq h_{\tau} (\sigma)=0\]
and the proof is finished.

\end{proof}

Note that considering a direct sum of the dynamical system satisfying the conditions in Theorem \ref{thm} with a noncommutative
dynamical system $(\blg, \gamma)$ such that $\hte(\gamma) = \hte_c(\gamma) \in (0, \frac{\log 2}{2})$ one can obtain examples in
which $\hte(\alpha) > \hte_c(\alpha) >0$.

One version of the analogous problem for the Connes-St\o rmer entropy would be the following -- let $\mlg$ be the hyperfinite
$II_1$-factor with the trace $\tau$. Does there exist a trace-preserving automorphism $\alpha \in \tu{Aut}(\mlg)$ such that
\[
h_{\tau} (\alpha) > \sup \left\{h_{\tau}(\alpha|_{\clg}): \clg - \alpha-\textrm{invariant } \textrm{ masa in } \mlg\right\}?\] In
\cite{GolNesh} the authors gave examples of automorphisms $\alpha$ of the hyperfinite $II_1$-factor such that $h_{\tau} (\alpha)
> h_{\tau}(\alpha|_{\clg})$, where $\clg$ is a fixed $\alpha$-invariant Cartan subalgebra of $\mlg$ (other examples of such situation
arise for permutation endomorphisms of Cuntz algebras, as can be deduced from \cite{LMP}), but the  problem stated above remains
open.

 In view of Theorem \ref{thm} it is also natural to ask the following question: do  there exist a (say unital)
$C^*$-algebra $\alg$ and $\alpha \in \Aut(\alg)$ such that $\hte(\alpha)>0$ and $\alg$ does not admit any nontrivial
$\alpha$-invariant abelian subalgebras?

\vspace*{0.5cm}

\noindent \emph{Acknowledgment.} The note was written during a visit of the author to University of Tokyo in October-November
2009 funded by a JSPS Short Term Postdoctoral Fellowship. The author would like to thank Sergey Neshveyev for useful comments on
the first draft.



\end{document}